\documentclass[11pt]{amsart}
\usepackage[T1]{fontenc}
\usepackage{lmodern}
\usepackage{amsmath,amssymb,mathtools}
\usepackage{enumitem}
\usepackage[numbers,sort&compress]{natbib}
\usepackage[margin=1.02in]{geometry}
\usepackage{microtype}
\usepackage[colorlinks=true,linkcolor=blue,citecolor=blue,urlcolor=blue]{hyperref}

\newtheorem{theorem}{Theorem}[section]
\newtheorem{proposition}[theorem]{Proposition}
\newtheorem{lemma}[theorem]{Lemma}
\newtheorem{corollary}[theorem]{Corollary}
\newtheorem{conjecture}[theorem]{Conjecture}
\theoremstyle{definition}

\theoremstyle{remark}

\newcommand{\R}{\mathbb{R}}
\newcommand{\C}{\mathbb{C}}
\newcommand{\Z}{\mathbb{Z}}
\newcommand{\Hh}{\mathbb{H}}

\newcommand{\dd}{\mathrm{d}}
\newcommand{\PSL}{\operatorname{PSL}}
\newcommand{\supp}{\operatorname{supp}}
\newcommand{\Div}{\operatorname{div}}

\newcommand{\capacity}{\operatorname{Cap}}

\title[Zariski-Dense Monodromy]
{Zariski-Dense Monodromy of Singular Hyperbolic Metrics on Non-Hyperbolic Riemann Surfaces}

\author{Yu Feng}
\address{College of Science, Tianjin University of Science and Technology, Tianjin 300457, China}
\email{yuf@tust.edu.cn}

\author{Yiqian Shi}
\address{CAS Wu Wen-Tsun Key Laboratory of Mathematics and School of Mathematical Sciences, University of Science and Technology of China, Hefei 230026, China}
\email{yqshi@ustc.edu.cn}

\author{Jijian Song}
\address{School of Mathematics and KL-AAGDM, Tianjin University, Tianjin 300350, China}
\email{jijian.song@tju.edu.cn}

\author{Bin Xu}
\address{CAS Wu Wen-Tsun Key Laboratory of Mathematics and School of Mathematical Sciences, University of Science and Technology of China, Hefei 230026, China}
\email{bxu@ustc.edu.cn}


\begin{document}

\begin{abstract}
We prove that the monodromy group of every singular hyperbolic metric on a non-hyperbolic Riemann surface, in the sense of potential theory, is Zariski dense in $\PSL(2,\R)$, confirming a conjecture of the authors~\cite{FengShiSongXu2023}. The main new step is to show that a singular hyperbolic metric on an arbitrary parabolic Riemann surface cannot have monodromy contained in a conjugate of the real affine subgroup of $\PSL(2,\R)$. The same argument also gives a direct proof in the compact case. Combined with the nonexistence results for the remaining proper subgroup types in \cite{FengShiSongXu2023}, this proves the conjecture.
\end{abstract}

\maketitle

\section{Introduction}

We study the monodromy of singular hyperbolic metrics on Riemann
surfaces, allowing the surface to be noncompact and the singular set to
be closed and discrete. Building on \cite{FengShiSongXu2023}, we examine
how the potential-theoretic type of the underlying surface constrains
the associated monodromy group.

A singular hyperbolic metric has a developing map whose analytic continuation gives a monodromy representation into $\PSL(2,\R)$, unique up to conjugacy. We prove that if the underlying Riemann surface is compact or parabolic, then its monodromy group is Zariski dense in
$\PSL(2,\R)$. The necessary definitions and background are recalled
below.

\subsection{Singular hyperbolic metrics}

The local classification of isolated singularities of conformal
hyperbolic metrics goes back to Nitsche \cite{Nitsche1957} and Heins
\cite{Heins1962}, who showed that every such singularity is either
conical or cuspidal. An elementary proof based on complex analysis,
together with explicit local normal forms for the metric, was given by
Feng--Shi--Xu \cite{FengShiXu2020}. The same dichotomy can also be
understood from the Higgs-bundle viewpoint. More precisely, Li and
Mochizuki showed that a hyperbolic metric on a punctured disc is mutually
bounded near the puncture with one of the standard cone or cusp models;
see \cite[Lemma~8.3]{LM2020}.

Let $\Sigma$ be a connected Riemann surface, not necessarily compact, and let $$
\mathrm{D}
=
\sum_{i\in I}(\theta_i-1)p_i
$$
be an $\R$-divisor on $\Sigma$, where the points $p_i$ are distinct and
$
\theta_i\geq0,\ 
\theta_i\neq1
$
for every $i\in I$. We assume that
$
\supp\mathrm{D}
:=
\{p_i:i\in I\}
$
is a closed discrete subset of $\Sigma$. In particular, the divisor is
locally finite, although the index set $I$ may be infinite when $\Sigma$
is noncompact.

A \emph{singular hyperbolic metric representing $\mathrm{D}$} is a
conformal metric $\dd s^2$ of Gaussian curvature $-1$ on
$
\Sigma\setminus\supp\mathrm{D}
$
satisfying the following local conditions. For each
$p_i\in\supp\mathrm{D}$, choose a coordinate neighborhood $U_i$ and a
holomorphic coordinate
$z_i: U_i\longrightarrow z_i(U_i)\subset\C,\ z_i(p_i)=0.$
On $U_i\setminus\{p_i\}$, write
$
\dd s^2=e^{2u_i}|\dd z_i|^2.
$
If $\theta_i>0$, then
$$
u_i-(\theta_i-1)\log|z_i|
$$
extends continuously across $p_i$. In this case, $p_i$ is called a cone
singularity of angle $2\pi\theta_i$. If $\theta_i=0$, then, after shrinking $U_i$ so that
$0<|z_i|<1$ on $U_i\setminus\{p_i\}$,
$$
u_i+\log|z_i|+\log\bigl(-\log|z_i|\bigr)
$$
extends continuously across $p_i$. In this case, $p_i$ is called a cusp
singularity.

For a compact connected Riemann surface, the existence and uniqueness of
a singular hyperbolic metric are governed by the
Gauss--Bonnet condition. Let
$$
\mathrm{D}
=
\sum_{i=1}^{n}(\theta_i-1)p_i,
\qquad
\theta_i\geq0.
$$
A hyperbolic metric representing $\mathrm{D}$ can exist only if
$$
\chi(\Sigma)+\sum_{i=1}^{n}(\theta_i-1)<0,
$$
where $\chi(\Sigma)$ denotes the Euler characteristic of $\Sigma$. If
$\Sigma$ has genus $g_\Sigma$, then
$
\chi(\Sigma)=2-2g_\Sigma.
$

Using the theory of $S$-$K$ metrics and a Perron-type extremal
construction, Heins proved that this condition is also sufficient and
that the corresponding hyperbolic metric is unique
\cite[Chapter~II]{Heins1962}. For conical singularities, independent PDE
proofs were subsequently obtained by McOwen and Troyanov
\cite{McOwen1988,Troyanov1991}. Prescribed-curvature problems on open
surfaces were studied by Hulin--Troyanov
\cite{HulinTroyanov1992}, while further existence results for conformal
metrics with prescribed curvature and singularities were established by
McOwen \cite{McOwen1993}. These works primarily concern the existence
and uniqueness of the metric itself. In the present paper, we study
instead the algebraic size of the monodromy group of its developing map.

\subsection{Developing maps and monodromy}

Let
$
\Sigma^*
=
\Sigma\setminus\supp\mathrm{D}
$
be the regular part of the metric. A singular hyperbolic metric
$\dd s^2$ admits a holomorphic locally univalent developing map from the
universal cover of $\Sigma^*$ to the upper half-plane $\Hh$, and the
developing map is equivariant with respect to a representation
$$
\rho:\pi_1(\Sigma^*)\longrightarrow\PSL(2,\R).
$$
Changing the developing map conjugates the image of $\rho$. Consequently,
the conjugacy class of
$
\rho\bigl(\pi_1(\Sigma^*)\bigr)
$
is determined by the metric; see
\cite[Theorem~1.2]{FengShiSongXu2023}. Throughout this paper, we refer to
this conjugacy class as the \emph{monodromy group} of $\dd s^2$.

Equivalently, the developing map may be regarded as a multivalued locally
univalent holomorphic map on $\Sigma^*$ whose branches are related by
real M\"obius transformations. The local singularity type determines the
conjugacy class of the corresponding peripheral monodromy. The precise
developing-map formalism and the subgroup reduction used in the proof are
given in Section~\ref{sec:geometric-setup}.

The restrictions on the monodromy considered here depend on the
potential-theoretic properties of the underlying surface. We recall next
the notions needed to formulate this condition.

\subsection{Potential-theoretic type of the base surface}

We use the potential-theoretic classification of Riemann surfaces. A
Riemann surface is called \emph{elliptic} if it is compact; a noncompact
surface is called \emph{parabolic} if it carries no nonconstant negative
subharmonic function, and \emph{hyperbolic} otherwise. We use
\emph{non-hyperbolic} for elliptic and parabolic surfaces; see
\cite[p.~179]{FarkasKra} and
\cite[Chapter~IV, \S1.6, pp.~203--206]{AhlforsSario}.

For our proof, the essential property of a parabolic surface is the
existence of compactly supported cutoff functions converging locally to
one whose Dirichlet energies tend to zero. The precise form
needed here is proved in Lemma~\ref{lem:parabolic-cutoff}. These cutoffs
play the role of the constant function $1$ in the compact case.

This potential-theoretic viewpoint was applied to singular hyperbolic
metrics in \cite[Definition~1.3 and Theorem~1.6]{FengShiSongXu2023}.
Their results exclude several proper subgroup types and reduce the
remaining general case to affine monodromy.

\subsection{Proper subgroups and the remaining affine case} We recall the subgroup notation used in \cite[Proposition~1.4]{FengShiSongXu2023}. The affine subgroup is $$ L = \left\{ \begin{pmatrix} a&b\\ 0&a^{-1} \end{pmatrix} : a>0,\ b\in\R \right\}, $$ whose conjugates are precisely the stabilizers of points of $\partial\Hh$. For $c>0$, let $$ H_{1c} = \left\{ \begin{pmatrix} c^n&t\\ 0&c^{-n} \end{pmatrix} : n\in\Z,\ t\in\R \right\}, $$ and in particular $$ L_0 = H_{11} = \left\{ \begin{pmatrix} 1&t\\ 0&1 \end{pmatrix} : t\in\R \right\}. $$ The split torus is $$ H_2 = \left\{ \begin{pmatrix} a&0\\ 0&a^{-1} \end{pmatrix} : a>0 \right\}, $$ with normalizer $$ H_2' = H_2 \cup \left\{ \begin{pmatrix} 0&b\\ -b^{-1}&0 \end{pmatrix} : b>0 \right\}. $$ Finally, $$ H_3 = \left\{ \begin{pmatrix} \cos t&\sin t\\ -\sin t&\cos t \end{pmatrix} : t\in\R \right\} /\{\pm I_2\} $$ is a maximal compact subgroup of $\PSL(2,\R)$. As shown precisely in Proposition~\ref{prop:subgroup-reduction}, if a subgroup of $\PSL(2,\R)$ is not Zariski dense, then, after conjugation, it is contained in one of the three groups $L,\ H_2',\ H_3.$

Following \cite{FengShiSongXu2023}, a singular hyperbolic metric is called a $G$-metric if its monodromy group is contained in a conjugate of $G$. The authors proved that a non-hyperbolic Riemann surface carries no singular hyperbolic $H_2$-, $H_2'$-, $H_3$-, or $L_0$-metric \cite[Theorem~1.6]{FengShiSongXu2023}, and excluded $L$-metrics for several particular non-hyperbolic surfaces \cite[Corollary~5.4 and Theorem~5.6]{FengShiSongXu2023}. They formulated the following conjecture. \begin{conjecture}[Feng--Shi--Song--Xu] \label{conj:zariski-intro} The monodromy group of a singular hyperbolic metric on a non-hyperbolic Riemann surface is Zariski dense in $\PSL(2,\R)$. \end{conjecture} In view of the preceding subgroup exclusions, the remaining general case is the affine one. Our principal result is therefore the following. \begin{theorem} \label{thm:no-affine-intro} Let $\Sigma$ be a compact or parabolic Riemann surface, and let $\dd s^2$ be a singular hyperbolic metric on $\Sigma$. Then the monodromy group of $\dd s^2$ is not contained in any conjugate of $L$. \end{theorem} Together with the results of \cite{FengShiSongXu2023}, this yields the main theorem. \begin{theorem} \label{thm:main-intro} The monodromy group of every singular hyperbolic metric on a non-hyperbolic Riemann surface is Zariski dense in $\PSL(2,\R)$. \end{theorem} The subgroup reduction used above is proved in Proposition~\ref{prop:subgroup-reduction}.

\subsection{Idea of the proof and organization of the paper}

We briefly explain the main idea behind
Theorem~\ref{thm:no-affine-intro}. Suppose that the monodromy is
contained in the affine subgroup $L$, and write the developing map on
the universal cover as
$\widetilde f=\widetilde u+{\rm i}\widetilde v,\ \widetilde v>0.$
Affine monodromy multiplies $\widetilde v$ by a positive constant.
Therefore
$\dd\log\widetilde v$
is invariant under deck transformations and descends to a globally
defined real one-form $\alpha$ on $\Sigma^*$. For a local branch
$f=u+{\rm i}v$ of the developing map,
$$
\alpha=\dd\log v.
$$
Since $v$ is positive and harmonic, $\alpha$ satisfies
$$
\Div_g(\alpha^{\sharp_g})=-|\alpha|_g^2
$$
for any smooth metric $g$ compatible with the complex structure.
The main analytic point is to justify the global use of this identity
in the presence of cone and cusp singularities. Local estimates near
the singular set provide the required integrability and show that the
boundary terms arising from small circles around the singularities
vanish.

Applying the divergence theorem then gives, for every
$\phi\in C_c^\infty(\Sigma)$,
$$
\int_{\Sigma^*}\phi^2|\alpha|_g^2\,\dd A_g
\leq
4\int_{\Sigma}|\dd\phi|_g^2\,\dd A_g.
$$
If $\Sigma$ is compact, we take $\phi\equiv1$. If $\Sigma$ is
noncompact and parabolic, we use compactly supported cutoff functions
whose Dirichlet energies tend to zero. In both cases the inequality
forces
$$
\alpha\equiv0.
$$
Hence the imaginary part of the developing map is locally constant;
holomorphicity then forces the developing map itself to be locally
constant, contradicting its local univalence. This excludes affine
monodromy.

The paper is organized as follows.
Section~\ref{sec:geometric-setup} recalls developing maps and establishes the subgroup reduction to the affine case.
Section~\ref{sec:differential-identity} constructs $\alpha$ and proves
the divergence identity.
Section~\ref{sec:local-estimates} establishes the local estimates near
cone and cusp singularities.
Section~\ref{sec:energy-inequality} proves the compactly supported
energy inequality.
Finally, Section~\ref{sec:parabolic-cutoffs} applies compactness or
parabolic cutoffs to exclude affine monodromy and completes the proof
of the main theorem.

\section{Developing maps and reduction to the affine case}
\label{sec:geometric-setup}

In this section we recall developing maps and monodromy, and reduce the
main theorem to the case of affine monodromy.

\subsection{Developing maps and monodromy}

Let $\supp\mathrm{D}=\{p_i\}_{i\in I}\subset\Sigma$
be a closed discrete subset and set
$\Sigma^*=\Sigma\setminus \supp\mathrm{D}.$
Since $\supp\mathrm{D}$ is closed and discrete, it is locally finite.

Let $\dd s^2$ be a conformal metric of Gaussian curvature $-1$ on
$\Sigma^*$, with a cone or cusp singularity at each point of $\supp\mathrm{D}$.
Write
$\Hh=\{w\in\C:\Im w>0\},\ \dd s_{\Hh}^2=\frac{|\dd w|^2}{(\Im w)^2}.$
If $\pi:\widetilde{\Sigma^*}\longrightarrow\Sigma^*$
is the universal covering map, then $\dd s^2$ admits a holomorphic
locally univalent developing map
$$
\widetilde f:\widetilde{\Sigma^*}\longrightarrow\Hh
$$
such that $\pi^*\dd s^2=\widetilde f^*\dd s_{\Hh}^2.$
There is a representation
$$
\rho:\pi_1(\Sigma^*)\longrightarrow\PSL(2,\R)
$$
for which $\widetilde f$ is $\rho$-equivariant:
$\widetilde f(\gamma\cdot\widetilde x)=\rho(\gamma)\bigl(\widetilde f(\widetilde x)\bigr).$

The developing map is unique up to postcomposition by an element of
$\PSL(2,\R)$. Replacing $\widetilde f$ by
$A\circ\widetilde f$ replaces $\rho$ by
$A\rho A^{-1}$. Hence the conjugacy class $\left[\rho\bigl(\pi_1(\Sigma^*)\bigr)\right]$
is independent of the choice of developing map; see
\cite[Theorem~1.2]{FengShiSongXu2023}. We call this conjugacy class the
\emph{monodromy group} of $\dd s^2$.

We shall also use the associated multivalued developing map on the base
surface. If $U\subset\Sigma^*$ is simply connected and evenly covered
by $\pi$, a local section
$s_U:U\to\widetilde{\Sigma^*}$ gives a branch
$$
f_U=\widetilde f\circ s_U:U\longrightarrow\Hh.
$$
Two such branches differ by postcomposition with a monodromy
transformation. Thus $\widetilde f$ is the single-valued developing map
on the universal cover, whereas
$$
f:\Sigma^*\longrightarrow\Hh
$$
will denote the corresponding multivalued developing map on the base
surface.

If $G\subset\PSL(2,\R)$, we say that the monodromy group is
\emph{contained in a conjugate of $G$} if the image of a monodromy
representation is contained in some conjugate of $G$. This condition is
independent of the chosen developing map.

\subsection{Subgroup reduction}
We now justify the subgroup reduction stated in the introduction. We use the notation $L$, $H_{1c}$, $H_2$, $H_2'$, and $H_3$
introduced there. Following \cite{FengShiSongXu2023}, we use
Zariski density in $\PSL(2,\R)$ in the equivalent sense that a subgroup is Zariski dense if it is not contained in any positive-dimensional proper Lie subgroup of $\PSL(2,\R)$. We also use the classification of such Lie subgroups given in \cite[Proposition~1.4]{FengShiSongXu2023}.

\begin{proposition}[Subgroup reduction]
\label{prop:subgroup-reduction}
Let $\Gamma\subset\PSL(2,\R)$. If $\Gamma$ is not Zariski dense in
$\PSL(2,\R)$, then, up to conjugacy,
$$
\Gamma\subset L,\qquad
\Gamma\subset H_2',\qquad\text{or}\qquad
\Gamma\subset H_3.
$$
\end{proposition}

\begin{proof}
Since $\Gamma$ is not Zariski dense, it is contained in a
positive-dimensional proper Lie subgroup $G$ of $\PSL(2,\R)$.
By \cite[Proposition~1.4]{FengShiSongXu2023}, after conjugation,
$G$ is one of
$$
L,\qquad H_{1c},\qquad H_2,\qquad H_2',\qquad H_3.
$$
Since $H_{1c}\subset L$ and $H_2\subset H_2'$, the five
possibilities reduce to
$$
\Gamma\subset L,\qquad
\Gamma\subset H_2',\qquad\text{or}\qquad
\Gamma\subset H_3.
$$
\end{proof}

By the nonexistence results recalled in the introduction, the only remaining general case is affine monodromy. Hence, after postcomposing the developing map by a suitable element of $\PSL(2,\R)$, we may assume throughout the rest of the proof that 
$$ \rho\bigl(\pi_1(\Sigma^*)\bigr)\subset L. $$ The following sections exclude this possibility when the underlying Riemann surface is compact or parabolic.

\section{The global differential associated with affine monodromy}
\label{sec:differential-identity}

Throughout this section, we assume that the image of the monodromy
representation is contained in the affine subgroup $L$. Thus, for every
deck transformation $\gamma$, there exist constants
$\lambda_\gamma>0$ and $\mu_\gamma\in\R$ such that
$\widetilde f\circ\gamma=\lambda_\gamma\widetilde f+\mu_\gamma.$

\subsection{Construction of the global one-form}

Write the developing map as
$\widetilde f=\widetilde u+{\rm i}\widetilde v$, where
$\widetilde v>0$. Taking imaginary parts in the affine equivariance
relation gives $\widetilde v\circ\gamma=\lambda_\gamma\widetilde v.$
Consequently,
$$
\log(\widetilde v\circ\gamma)
=
\log\widetilde v+\log\lambda_\gamma,
$$
and therefore $\gamma^*(\dd\log\widetilde v)=\dd\log\widetilde v.$
Thus $\dd\log\widetilde v$ is invariant under all deck transformations
and descends to a unique smooth real one-form $\alpha$ on $\Sigma^*$,
characterized by
$$
\pi^*\alpha
=
\dd\log\widetilde v.
$$
In particular, $\alpha$ is closed.

This construction has a simple description in terms of local branches.
Let $U\subset\Sigma^*$ be a simply connected evenly covered open set,
let $s_U:U\to\widetilde{\Sigma^*}$ be a local section of $\pi$, and set
$$
f_U
=
\widetilde f\circ s_U
=
u_U+{\rm i}v_U.
$$
Pulling back the defining identity for $\alpha$ by $s_U$, we obtain
$\alpha|_U=\dd\log v_U.$

If a different local section is chosen, then the corresponding imaginary
part is a positive constant multiple of $v_U$. Hence its logarithm differs
from $\log v_U$ by a constant, and the local expression for $\alpha$ is
unchanged.

Fix a smooth Riemannian metric $g$ compatible with the complex structure
of $\Sigma$. We write $\dd A_g$ and $\dd s_g$ for the corresponding area
and arclength elements, and use the convention
$\Delta_g=\Div_g\nabla_g$.

We shall repeatedly use the conformal invariance of the relevant
two-dimensional quantities. If $\widehat g=e^{2\sigma}g$, then, for every
real one-form $\beta$,
$|\beta|_{\widehat g}^2\,\dd A_{\widehat g}=|\beta|_g^2\,\dd A_g.$
Along a smooth curve, with compatibly oriented unit normals,
$\beta(\nu_{\widehat g})\,\dd s_{\widehat g}=\beta(\nu_g)\,\dd s_g.$
Moreover, for every smooth function $h$, $\Delta_{\widehat g}h=e^{-2\sigma}\Delta_gh,\
|\dd h|_{\widehat g}^2=e^{-2\sigma}|\dd h|_g^2.$

\subsection{The nonlinear divergence identity}

The global one-form $\alpha$ satisfies a nonlinear equation that will be
the main analytic input in the proof.

\begin{lemma}[Nonlinear divergence identity]
\label{lem:global-divergence}
For every smooth Riemannian metric $g$ compatible with the complex
structure of $\Sigma$, the one-form $\alpha$ satisfies
$$
\Div_g(\alpha^{\sharp_g})
=
-|\alpha|_g^2
$$
on $\Sigma^*$, where $\alpha^{\sharp_g}$ is the vector field metrically
dual to $\alpha$.
\end{lemma}

\begin{proof}
Let $U\subset\Sigma^*$ be a simply connected coordinate neighborhood and
write a local branch of the developing map as
$f_U=u_U+{\rm i}v_U$, where $v_U>0$. Set $h_U=\log v_U$. Then
$\dd h_U=\alpha|_U$.

Since $f_U$ is holomorphic, $v_U$ is harmonic. Hence, in a holomorphic
coordinate,
$$\Delta_0 h_U=\Delta_0\log v_U
=-|\dd v_U|_0^2/v_U^2=-|\dd h_U|_0^2.$$
If $g=e^{2\sigma}|\dd z|^2$, then in real dimension two, we have
$\Delta_g h_U=e^{-2\sigma}\Delta_0h_U$ and
$|\dd h_U|_g^2=e^{-2\sigma}|\dd h_U|_0^2$. Therefore
$\Delta_g h_U=-|\dd h_U|_g^2$.

Since $\alpha|_U=\dd h_U$, we obtain
$\Div_g(\alpha^{\sharp_g})=\Delta_g h_U=-|\alpha|_g^2$.
As $U$ is arbitrary, the identity holds on all of $\Sigma^*$.
\end{proof}

\section{Local behavior at the singular set}
\label{sec:local-estimates}

We now study the behavior of the global one-form $\alpha$ near the
singular set $\supp\mathrm{D}$. Two properties will be needed in the global
integration argument: local square-integrability of $\alpha$ across the
singular points and the vanishing of the boundary flux along shrinking
coordinate circles.

In real dimension two, both the energy measure $|\alpha|_g^2\,\dd A_g$
and the flux measure $\alpha(\nu_g)\,\dd s_g$
are invariant under conformal changes of the background metric.
Consequently, the required local estimates may be carried out with
respect to the Euclidean metric in a holomorphic coordinate.

Assume throughout that the monodromy is contained in the affine subgroup
$L$. By \cite[Lemma~5.3 and its proof]{FengShiSongXu2023}, every
singularity is either a cusp or a cone singularity of angle
$2\pi m$, $m\in\Z_{>1}$, and the corresponding local forms of the
developing map will be used below.

\subsection{Cusp singularities}

Let $p\in \supp\mathrm{D}$ be a cusp singularity. By
\cite[Lemma~5.3 and its proof]{FengShiSongXu2023}, there exist a
holomorphic coordinate $z$ centered at $p$, a radius $0<R_p<1$, and
constants $\lambda>0$ and $s\in\R$ such that, on every simply connected domain $W\subset\{0<|z|<R_p\},$
a local branch of the developing map is
$$
f_{p,W}(z)=u_{p}(z)+{\rm i}v_{p}(z)=
-{\rm i}\lambda\operatorname{Log}_W z+s.
$$
Writing $\operatorname{Log}_W z=\log r+{\rm i}\vartheta,\ r=|z|,$ we obtain
$f_{p,W}(z)=\lambda\vartheta+s-{\rm i}\lambda\log r.$
Hence $\Im f_{p,W}(z)=-\lambda\log r.$ Two branches of the logarithm differ by $2\pi {\rm i}k$, $k\in\Z$, so the
corresponding branches of the developing map differ only by the real
constant $2\pi\lambda k$. Their imaginary parts therefore agree.
Consequently,
$v_p(z):=-\lambda\log|z|$
defines a single-valued positive function on
$\{0<|z|<R_p\}$. Thus
$$
\alpha=\dd\log v_p=\frac{\dd r}{r\log r}.
$$

The estimates needed later are summarized in the following proposition.

\begin{proposition}[Cusp estimates]
\label{prop:cusp}
For every $0<R\leq R_p$, where $0<R_p<1$
$$
\int_{0<|z|<R}
|\alpha|_0^2\,\dd x\,\dd y
=
\frac{2\pi}{|\log R|}.
$$
Moreover, let $\phi\in C^\infty(\{|z|<R_p\})$ be bounded. For
$0<\varepsilon<R_p$, set
$A_{\varepsilon,R_p}=\{z\in\C:\varepsilon<|z|<R_p\},$
and let $\nu_\varepsilon$ denote the outward Euclidean unit normal of
$A_{\varepsilon,R_p}$ along its inner boundary
$\{|z|=\varepsilon\}$. Then
$$
\left|
\int_{|z|=\varepsilon}
\phi^2\alpha(\nu_\varepsilon)\,\dd s_0
\right|
\leq
\frac{
2\pi\|\phi\|_{L^\infty(|z|<R_p)}^2
}{
|\log\varepsilon|
}
\longrightarrow0
$$
as $\varepsilon\to0$.
\end{proposition}

\begin{proof}
With respect to the Euclidean metric in polar coordinates,
$\dd s_0^2=\dd r^2+r^2\dd\vartheta^2,$
we have $|\dd r|_0=1$. Since
$\alpha=\frac{\dd r}{r\log r},$
it follows that
$|\alpha|_0^2=\frac{1}{r^2(\log r)^2}.$
Using
$\dd x\,\dd y=r\,\dd r\,\dd\vartheta,$
we obtain
$$
\begin{aligned}
\int_{0<|z|<R}
|\alpha|_0^2\,\dd x\,\dd y
&=
\int_0^{2\pi}\int_0^R
\frac{1}{r^2(\log r)^2}
\,r\,\dd r\,\dd\vartheta\\
&=
2\pi\int_0^R
\frac{\dd r}{r(\log r)^2}
=
\frac{2\pi}{|\log R|}.
\end{aligned}
$$

Along the inner boundary $\{|z|=\varepsilon\}$ of
$A_{\varepsilon,R_p}$, the outward Euclidean unit normal is
$\nu_\varepsilon=-\partial_r.$
Since $\log\varepsilon<0$,
$$
\alpha(\nu_\varepsilon)
=
-\frac{1}{\varepsilon\log\varepsilon}
=
\frac{1}{\varepsilon|\log\varepsilon|}.
$$
Also,
$\dd s_0=\varepsilon\,\dd\vartheta.$
Therefore
$$
\begin{aligned}
\left|
\int_{|z|=\varepsilon}
\phi^2\alpha(\nu_\varepsilon)\,\dd s_0
\right|
&\leq
\|\phi\|_{L^\infty(|z|<R_p)}^2
\int_0^{2\pi}
\frac{1}{\varepsilon|\log\varepsilon|}
\,\varepsilon\,\dd\vartheta\\
&=
\frac{
2\pi\|\phi\|_{L^\infty(|z|<R_p)}^2
}{
|\log\varepsilon|
},
\end{aligned}
$$
which tends to zero as $\varepsilon\to0$.
\end{proof}

\subsection{Integral cone singularities}

We first record the precise local form of the developing map near an integral cone point. The following lemma shows, in particular, that the peripheral monodromy is trivial and hence that the developing map admits a single-valued holomorphic branch on a punctured coordinate disc. This form will be used in the cone-point estimates below.

\begin{lemma}[Local form at an integral cone point]
\label{lem:cone-extension}
Let $p\in \supp\mathrm{D}$ be a cone singularity of angle $2\pi m$, where
$m\geq2$. There exist a holomorphic coordinate $z$ centered at $p$, a
radius $R_0>0$, and a single-valued local branch $f_p:\{0<|z|<R_0\}\longrightarrow\Hh$
of the developing map such that
$$
f_p(z)
=
\frac{az^m+b}{cz^m+d},
\qquad
ad-bc\neq0.
$$
After decreasing $R_0$ if necessary, $f_p$ extends holomorphically
across $z=0$, its extended value satisfies
$f_p(0)=\frac{b}{d}\in\Hh,$
and
$f_p'(z)=z^{m-1}g_p(z),$
where $g_p$ is holomorphic near $0$ and $g_p(0)\neq0$.
\end{lemma}

\begin{proof}
The local expression is obtained in
\cite[Lemma~5.3 and its proof]{FengShiSongXu2023}. We only make explicit
the behavior at the origin.

First, $d\neq0$. Otherwise, $ad-bc\neq0$ implies $b\neq0$ and
$c\neq0$, and, after shrinking the coordinate disc if necessary,
$$
G(z)
=
\frac{1}{f_p(z)}
=
\frac{cz^m}{az^m+b}
$$
extends holomorphically across $z=0$ with $G(0)=0$. For $z\neq0$,
because $f_p(z)\in\Hh$,
$\Im G(z)=-\frac{\Im f_p(z)}{|f_p(z)|^2}<0.$
Thus the holomorphic extension of $G$ maps a neighborhood of $0$ into
the closed lower half-plane and sends the interior point $0$ to its
boundary. Since $G$ is nonconstant, this contradicts the open mapping
theorem. Hence $d\neq0$.

It follows that the rational expression extends holomorphically across
$z=0$, with $f_p(0)=\frac{b}{d}.$
Since $\Im f_p(z)>0$ for $z\neq0$, continuity gives
$\Im f_p(0)\geq0$. If $\Im f_p(0)=0$, then the nonconstant
holomorphic map $f_p$ would map a neighborhood of $0$ into the closed
upper half-plane and send the interior point $0$ to its boundary, again
contradicting the open mapping theorem. Therefore
$$
f_p(0)=\frac{b}{d}\in\Hh.
$$
Finally,
$$
f_p'(z)
=
\frac{m(ad-bc)z^{m-1}}{(cz^m+d)^2}.
$$
Thus
$f_p'(z)=z^{m-1}g_p(z),\
g_p(z)=\frac{m(ad-bc)}{(cz^m+d)^2}.$
Since $d\neq0$, the function $g_p$ is holomorphic near the origin and
$g_p(0)=\frac{m(ad-bc)}{d^2}\neq0.
$
\end{proof}

\begin{proposition}[Cone estimate]
\label{prop:cone}
Let $p\in \supp\mathrm{D}$ be a cone singularity of angle $2\pi m$, where
$m\geq2$, and write the local branch of
Lemma~\ref{lem:cone-extension} as
$f_p=u_p+{\rm i}v_p.$
Then there exist constants $R_p>0$ and $C_p>0$ such that
$$
|\alpha(z)|_0
\leq
C_p|z|^{m-1}
$$
for every $0<|z|\leq R_p$. Consequently,
$$
\int_{0<|z|<R_p}
|\alpha|_0^2\,\dd x\,\dd y
\leq
\frac{\pi C_p^2}{m}R_p^{2m}
<\infty.
$$
Moreover, let $\phi\in C^\infty(\{|z|<R_p\})$ be bounded. For
$0<\varepsilon<R_p$, set
$A_{\varepsilon,R_p}=\{z\in\C:\varepsilon<|z|<R_p\},$
and let $\nu_\varepsilon$ denote the outward Euclidean unit normal of
$A_{\varepsilon,R_p}$ along its inner boundary
$\{|z|=\varepsilon\}$. Then
$$
\left|
\int_{|z|=\varepsilon}
\phi^2\alpha(\nu_\varepsilon)\,\dd s_0
\right|
\leq
2\pi C_p
\|\phi\|_{L^\infty(|z|<R_p)}^2
\varepsilon^m,
$$
and hence the boundary flux tends to zero as
$\varepsilon\to0$.
\end{proposition}

\begin{proof}
By Lemma~\ref{lem:cone-extension}, the function
$f_p=u_p+{\rm i}v_p$
extends holomorphically across $z=0$, and
$v_p(0)=\Im f_p(0)>0.$
Moreover,
$f_p'(z)=z^{m-1}g_p(z),$
where $g_p$ is holomorphic near the origin.

By continuity of $v_p$, after decreasing the coordinate radius if
necessary, choose $R_p>0$ such that
$$
v_p(z)
\geq
c_p
:=
\frac12v_p(0)
>0
$$
for all $|z|\leq R_p$. We also choose $R_p$ so that $g_p$ is
holomorphic on a neighborhood of the closed disc
$\{|z|\leq R_p\}$. Define
$M_p=\max_{|z|\leq R_p}|g_p(z)|$
and
$C_p=\frac{M_p}{c_p}.$

On the punctured disc,
$\alpha=\dd\log v_p=\frac{\dd v_p}{v_p}.$
Since $f_p=u_p+{\rm i}v_p$ is holomorphic, the Cauchy--Riemann equations give
$(u_p)_x=(v_p)_y,\ (u_p)_y=-(v_p)_x,$
and hence
$$
|\dd v_p|_0^2
=
(v_p)_x^2+(v_p)_y^2
=
|f_p'(z)|^2.
$$
Therefore
$$
\begin{aligned}
|\alpha(z)|_0&=
\frac{|\dd v_p|_0}{v_p(z)}
=
\frac{|f_p'(z)|}{v_p(z)}
=
\frac{|z|^{m-1}|g_p(z)|}{v_p(z)}\\
&\leq
\frac{M_p}{c_p}|z|^{m-1}=
C_p|z|^{m-1}.
\end{aligned}
$$
Using polar coordinates $z=re^{i\vartheta}$,
$$
\begin{aligned}
\int_{0<|z|<R_p}
|\alpha|_0^2\,\dd x\,\dd y
&\leq
\int_0^{2\pi}\int_0^{R_p}
C_p^2r^{2m-2}\,r\,\dd r\,\dd\vartheta\\
&=
2\pi C_p^2
\int_0^{R_p}r^{2m-1}\,\dd r\\
&=
\frac{\pi C_p^2}{m}R_p^{2m}.
\end{aligned}
$$
It remains to estimate the boundary flux. Along the inner boundary
$\{|z|=\varepsilon\}$ of $A_{\varepsilon,R_p}$, the outward Euclidean
unit normal points toward the deleted disc $\{|z|<\varepsilon\}$.
Thus
$\nu_\varepsilon
=
-\partial_r.
$
Since $\nu_\varepsilon$ has Euclidean norm one,
$$
|\alpha(\nu_\varepsilon)|
\leq
|\alpha|_0
\leq
C_p\varepsilon^{m-1}.
$$
Moreover, the Euclidean arclength element is
$\dd s_0=\varepsilon\,\dd\vartheta$. Hence
$$
\begin{aligned}
\left|
\int_{|z|=\varepsilon}
\phi^2\alpha(\nu_\varepsilon)\,\dd s_0
\right|
&\leq
\int_{|z|=\varepsilon}
|\phi|^2|\alpha(\nu_\varepsilon)|\,\dd s_0\\
&\leq
\|\phi\|_{L^\infty(|z|<R_p)}^2
\int_0^{2\pi}
C_p\varepsilon^{m-1}
\varepsilon\,\dd\vartheta\\
&=
2\pi C_p
\|\phi\|_{L^\infty(|z|<R_p)}^2
\varepsilon^m.
\end{aligned}
$$
Since $m\geq2$, the right-hand side tends to zero as
$\varepsilon\to0$.
\end{proof}

\subsection{Local square-integrability and boundary flux}

The preceding cusp and cone estimates give precisely the local
regularity needed for the integration-by-parts argument in the next
section.

\begin{corollary}[Local square-integrability and vanishing flux]
\label{cor:local-L2}
Extending $\alpha$ arbitrarily across the measure-zero set $\supp\mathrm{D}$, one has
$$
\alpha
\in
L_{\mathrm{loc}}^2(\Sigma;T^*\Sigma,g).
$$
Equivalently, for every compact set $K\Subset\Sigma$,
$$
\int_{K\cap\Sigma^*}
|\alpha|_g^2\,\dd A_g
<\infty.
$$
Furthermore, let $p\in \supp\mathrm{D}$ and choose a coordinate disc
$B_p(R_p)$ on which the corresponding cusp or cone estimate above is
valid. For $0<\varepsilon<R_p$, let
$\nu_{g,\varepsilon}$ denote the outward $g$-unit normal of $B_p(R_p)\setminus\overline{B_p(\varepsilon)}$
along its inner boundary $\partial B_p(\varepsilon)$. Then, for every smooth function $\phi$ bounded near $p$,
$$
\lim_{\varepsilon\to0}
\int_{\partial B_p(\varepsilon)}
\phi^2
\alpha(\nu_{g,\varepsilon})\,\dd s_g
=
0.
$$
\end{corollary}

\begin{proof}
Let $K\Subset\Sigma$ be compact. Since $\supp\mathrm{D}$ is closed and discrete,
$K\cap \supp\mathrm{D}$ is finite. The one-form $\alpha$ is smooth away from these
finitely many points. Near each cusp, Proposition~\ref{prop:cusp} gives
finite Euclidean energy, while near each cone point,
Proposition~\ref{prop:cone} gives finite Euclidean energy.

If $g=e^{2\sigma}g_0$ in a local holomorphic coordinate, then
$$
|\alpha|_g^2\,\dd A_g
=
|\alpha|_0^2\,\dd A_0.
$$
Thus the Euclidean energy estimates obtained above are exactly the
corresponding energy estimates for the background metric $g$. Since only
finitely many singular points meet $K$, summing the local contributions
gives
$$
\int_{K\cap\Sigma^*}
|\alpha|_g^2\,\dd A_g
<\infty.
$$

For the boundary-flux statement, in the same local coordinate write
$g=e^{2\sigma}g_0$. Along $\partial B_p(\varepsilon)$, let
$\nu_{0,\varepsilon}$ be the Euclidean unit normal with the same
orientation as $\nu_{g,\varepsilon}$. Then
$\nu_{g,\varepsilon}=e^{-\sigma}\nu_{0,\varepsilon},\ \dd s_g=e^\sigma\dd s_0.$
Consequently,
$$
\alpha(\nu_{g,\varepsilon})\,\dd s_g
=
\alpha(\nu_{0,\varepsilon})\,\dd s_0.
$$
The desired limit now follows directly from the flux estimates in
Propositions~\ref{prop:cusp} and \ref{prop:cone}.
\end{proof}

\section{The compactly supported energy inequality}
\label{sec:energy-inequality}

We now integrate the nonlinear divergence identity obtained in
Lemma~\ref{lem:global-divergence}. Since $\alpha$ is defined on
$\Sigma^*=\Sigma\setminus \supp\mathrm{D}$, we first remove small coordinate discs
around the singular points and apply the divergence theorem on the
resulting compact surface with boundary. The local estimates of
Section~\ref{sec:local-estimates} ensure that the boundary contributions
from the deleted discs vanish as their radii tend to zero.

\begin{proposition}[Compactly supported energy inequality]
\label{prop:energy-inequality}
For every real-valued function $\phi\in C_c^\infty(\Sigma)$,
$$
\int_{\Sigma^*}\phi^2|\alpha|_g^2\,\dd A_g
\leq
4\int_{\Sigma}|\dd\phi|_g^2\,\dd A_g.
$$
\end{proposition}

\begin{proof}
Choose a relatively compact smooth domain $U\Subset\Sigma$ such that
$\supp\phi\Subset U$
and
$\partial U\cap \supp\mathrm{D}=\varnothing.$
If $\Sigma$ is compact and $\supp\phi=\Sigma$, we simply take
$U=\Sigma$, in which case $\partial U=\varnothing$. Otherwise, since
$\supp\phi$ is compactly contained in $U$, the function $\phi$ vanishes
on a neighborhood of $\partial U$.

Because $\overline U$ is compact and $\supp\mathrm{D}$ is closed and discrete, the set
$\supp\mathrm{D}\cap\overline U$ is finite. Since $\partial U\cap \supp\mathrm{D}=\varnothing$, we
may write
$\supp\mathrm{D}\cap\overline U=\supp\mathrm{D}\cap U=\{p_1,\ldots,p_N\}.$
If $N=0$, no truncation at the singular set is needed; the argument
below is applied directly to $\overline U$, and there are no inner
boundary terms. We therefore assume $N\geq1$.

Choose pairwise disjoint coordinate discs
$$
B_j(R_j)=\{|z_j|<R_j\},
\qquad
\overline{B_j(R_j)}\Subset U,
$$
centered at $p_j$, with the radii chosen sufficiently small that the
local estimates of Propositions~\ref{prop:cusp} and \ref{prop:cone}
apply. For
$0<\varepsilon<\min_{1\leq j\leq N}R_j,$
define
$$
M_\varepsilon
=
\overline U
\setminus
\bigcup_{j=1}^{N}B_j(\varepsilon),
$$
where $B_j(\varepsilon)=\{|z_j|<\varepsilon\}$ is the open coordinate
disc.

The set $M_\varepsilon$ is a compact smooth surface with boundary. Indeed,
it is a closed subset of the compact set $\overline U$, and the deleted
discs have pairwise disjoint closures contained in the interior of $U$.
Its boundary is
$\partial M_\varepsilon=\partial U\sqcup\bigsqcup_{j=1}^{N}\partial B_j(\varepsilon).$
All singular points have been removed, so $\alpha$ is smooth on a
neighborhood of $M_\varepsilon$ in $\Sigma^*$.

Set $X=\alpha^{\sharp_g}.$
By Lemma~\ref{lem:global-divergence},
$\Div_gX=-|\alpha|_g^2.$
Since the metric dual of $X$ is $\alpha$, the product rule for divergence
gives
$$
\Div_g(\phi^2X)
=
2\phi\langle\dd\phi,\alpha\rangle_g
-
\phi^2|\alpha|_g^2.
$$
Applying the divergence theorem to the smooth vector field $\phi^2X$ on
the compact manifold $M_\varepsilon$, we obtain
$$
\int_{M_\varepsilon}
\Div_g(\phi^2X)\,\dd A_g
=
\int_{\partial M_\varepsilon}
\phi^2\alpha(\nu_g)\,\dd s_g,
$$
where $\nu_g$ denotes the outward $g$-unit normal to
$\partial M_\varepsilon$. Hence
$$
\begin{aligned}
\int_{M_\varepsilon}
\phi^2|\alpha|_g^2\,\dd A_g
={}&
2\int_{M_\varepsilon}
\phi\langle\dd\phi,\alpha\rangle_g\,\dd A_g-
\int_{\partial M_\varepsilon}
\phi^2\alpha(\nu_g)\,\dd s_g.
\end{aligned}
$$
The contribution from the outer boundary $\partial U$ is zero because
$\phi$ vanishes on a neighborhood of $\partial U$. It remains to consider
the inner boundary circles.

Along $\partial B_j(\varepsilon)$, the outward normal of
$M_\varepsilon$ points toward the deleted disc
$B_j(\varepsilon)$. In the local Euclidean coordinate $z_j$, this is the
same orientation as the outward normal used in
Propositions~\ref{prop:cusp} and \ref{prop:cone}. Moreover, the flux
measure of a one-form is conformally invariant in real dimension two:
$\alpha(\nu_g)\,\dd s_g=\alpha(\nu_0)\,\dd s_0.$
We may therefore apply the Euclidean flux estimates obtained in
Section~\ref{sec:local-estimates}.

If $p_j$ is a cusp, Proposition~\ref{prop:cusp} gives
$$
\left|
\int_{\partial B_j(\varepsilon)}
\phi^2\alpha(\nu_g)\,\dd s_g
\right|
\leq
\frac{
2\pi\|\phi\|_{L^\infty(B_j(R_j))}^2
}{
|\log\varepsilon|
},
$$
which tends to zero as $\varepsilon\to0$. If $p_j$ is a cone point of
angle $2\pi m_j$, Proposition~\ref{prop:cone} gives
$$
\left|
\int_{\partial B_j(\varepsilon)}
\phi^2\alpha(\nu_g)\,\dd s_g
\right|
\leq
2\pi C_{p_j}
\|\phi\|_{L^\infty(B_j(R_j))}^2
\varepsilon^{m_j},
$$
which also tends to zero. Since only finitely many singular points occur
in $U$, it follows that
$$
\int_{\partial M_\varepsilon}
\phi^2\alpha(\nu_g)\,\dd s_g
\longrightarrow0
$$
as $\varepsilon\to0$.

We next pass to the limit in the volume terms. By
Corollary~\ref{cor:local-L2},
$\alpha\in L^2_{\mathrm{loc}}(\Sigma;T^*\Sigma,g).$
Since $\phi$ and $\dd\phi$ have compact support in $U$, both
$\phi^2|\alpha|_g^2$
and
$|\phi|\,|\dd\phi|_g\,|\alpha|_g$
are integrable on $U$. Indeed, the latter follows from
Cauchy--Schwarz:
$$
\int_U
|\phi|\,|\dd\phi|_g\,|\alpha|_g\,\dd A_g
\leq
\left(
\int_U\phi^2|\alpha|_g^2\,\dd A_g
\right)^{1/2}
\left(
\int_U|\dd\phi|_g^2\,\dd A_g
\right)^{1/2}
<\infty.
$$
Choose a sequence $\varepsilon_k\downarrow0$. Then the sets
$M_{\varepsilon_k}$ increase to $\overline U\setminus \supp\mathrm{D}$. Since
$\partial U$ has zero two-dimensional measure, monotone convergence
applies to the nonnegative energy term, while dominated convergence
applies to the mixed term. Together with the vanishing of the boundary
flux, this yields
$$
\int_{U\setminus S}
\phi^2|\alpha|_g^2\,\dd A_g
=
2\int_{U\setminus S}
\phi\langle\dd\phi,\alpha\rangle_g\,\dd A_g.
$$
Since $\phi$ and $\dd\phi$ vanish outside $U$, this may be written as
$$
\int_{\Sigma^*}
\phi^2|\alpha|_g^2\,\dd A_g
=
2\int_{\Sigma^*}
\phi\langle\dd\phi,\alpha\rangle_g\,\dd A_g.
$$

The left-hand side is finite. Taking absolute values and applying the
Cauchy--Schwarz inequality, we obtain
$$
\begin{aligned}
\int_{\Sigma^*}
\phi^2|\alpha|_g^2\,\dd A_g
&\leq
2\int_{\Sigma^*}
|\phi|\,|\dd\phi|_g\,|\alpha|_g\,\dd A_g\\
&\leq
2
\left(
\int_{\Sigma^*}
\phi^2|\alpha|_g^2\,\dd A_g
\right)^{1/2}
\left(
\int_{\Sigma^*}
|\dd\phi|_g^2\,\dd A_g
\right)^{1/2}.
\end{aligned}
$$
Since the closed discrete set $\supp\mathrm{D}$ has zero area measure,
$\int_{\Sigma^*}|\dd\phi|_g^2\,\dd A_g=\int_{\Sigma}|\dd\phi|_g^2\,\dd A_g.$
If
$\int_{\Sigma^*}\phi^2|\alpha|_g^2\,\dd A_g=0,
$
the desired inequality is immediate. Otherwise, dividing by the square
root of this integral and then squaring gives
$$
\int_{\Sigma^*}
\phi^2|\alpha|_g^2\,\dd A_g
\leq
4\int_{\Sigma}
|\dd\phi|_g^2\,\dd A_g.
$$
This proves the proposition.
\end{proof}

\section{Parabolic cutoffs and exclusion of affine monodromy}
\label{sec:parabolic-cutoffs}

We now combine the compactly supported energy inequality with the
capacity characterization of parabolicity. The same energy estimate will
be used in both the compact and the noncompact parabolic cases, with
different choices of test functions.

\subsection{Capacity and zero-energy cutoff functions}

Let $\Sigma$ be a noncompact Riemann surface and let $g$ be a smooth
Riemannian metric compatible with its complex structure. For a compact
set $K\Subset\Sigma$, recall that its capacity relative to infinity is
defined by
$$
\capacity_{\Sigma}(K)
=
\inf\left\{
\int_{\Sigma}|\dd\psi|_g^2\,\dd A_g:
\psi\in C_c^\infty(\Sigma),\
\psi\geq1
\text{ on a neighborhood of }K
\right\}.
$$
The Dirichlet integral is conformally invariant in real dimension two,
so $\capacity_{\Sigma}(K)$ depends only on the conformal structure of
$\Sigma$. A noncompact Riemann surface is parabolic if and only if every
compact set with nonempty interior has zero capacity relative to infinity; see \cite[Corollary~3.1 and Theorem~4.2, pp.~130--131]
{TroyanovParabolicity}.

We shall use this characterization in the following form.

\begin{lemma}[Parabolic cutoff functions]
\label{lem:parabolic-cutoff}
Let $\Sigma$ be a noncompact parabolic Riemann surface. Then there exist
compact sets $K_j\subset\Sigma$ with smooth boundary such that
$K_j\Subset\operatorname{int}K_{j+1},\ \bigcup_{j=1}^{\infty}K_j=\Sigma,$
and functions $\phi_j\in C_c^\infty(\Sigma)$ satisfying
$$
0\leq\phi_j\leq1,
\qquad
\phi_j\equiv1
\text{ on a neighborhood of }K_j,
$$
and
$$
\int_{\Sigma}|\dd\phi_j|_g^2\,\dd A_g
<
\frac{1}{j}.
$$
Consequently, for every compact set $K\Subset\Sigma$, one has
$\phi_j\equiv1$ on $K$ for all sufficiently large $j$.
\end{lemma}

\begin{proof}
By \cite[Proposition~2.28, p.~46]{LeeSmooth}, choose a smooth positive
exhaustion function
$$
\tau:\Sigma\longrightarrow[0,\infty).
$$
Thus $\tau^{-1}([0,c])$ is compact for every $c<\infty$. Since $\Sigma$ is noncompact,
$\tau$ is necessarily unbounded. By Sard's theorem, the regular values
of $\tau$ are dense. Hence we may choose a strictly increasing sequence
of regular values
$$
c_1<c_2<c_3<\cdots,
\qquad
c_j\longrightarrow\infty,
$$
with $c_1$ sufficiently large that the corresponding sublevel set has
nonempty interior, and define
$$
K_j:=\tau^{-1}([0,c_j])=\{x\in\Sigma:\tau(x)\leq c_j\}.
$$
Each $K_j$ is compact by the exhaustion property. Since $c_j$ is a
regular value,
$\partial K_j=\tau^{-1}(c_j)$
is smooth, and $\operatorname{int}K_j=\{\tau<c_j\}.$
Since $c_j<c_{j+1}$,
$K_j\Subset\{\tau<c_{j+1}\}=\operatorname{int}K_{j+1}.$
Moreover, $c_j\to\infty$ implies $\bigcup_{j=1}^{\infty}K_j=\Sigma.$
Since $\Sigma$ is parabolic,
$\capacity_{\Sigma}(K_j)=0.$
Hence for each $j$ there exists
$\psi_j\in C_c^\infty(\Sigma)$ such that
$\psi_j\geq1$
on a neighborhood of $K_j$ and
$$
\int_{\Sigma}|\dd\psi_j|_g^2\,\dd A_g
<
\frac{1}{4j}.
$$

Choose once and for all a smooth nondecreasing function
$T:\R\longrightarrow[0,1]$
such that
$$
T(t)=0
\quad\text{for }t\leq0,
\qquad
T(t)=1
\quad\text{for }t\geq1,
$$
and
$0\leq T'(t)\leq2.$
Set
$$
\phi_j
=
T\circ\psi_j.
$$
Since $\psi_j$ is compactly supported and $T(0)=0$, the function
$\phi_j$ belongs to $C_c^\infty(\Sigma)$. Moreover,
$0\leq\phi_j\leq1,$
and $\phi_j\equiv1$ on a neighborhood of $K_j$. By the chain rule,
$\dd\phi_j=T'(\psi_j)\dd\psi_j,$
and therefore
$$
\begin{aligned}
\int_{\Sigma}|\dd\phi_j|_g^2\,\dd A_g
\leq
4\int_{\Sigma}|\dd\psi_j|_g^2\,\dd A_g<
\frac{1}{j}.
\end{aligned}
$$

Finally, let $K\Subset\Sigma$ be compact. Since $\tau$ is bounded on
$K$ and $c_j\to\infty$, one has $K\subset K_j$ for all sufficiently
large $j$. Hence $\phi_j\equiv1$ on $K$ for all such $j$.
\end{proof}

\subsection{Exclusion of affine monodromy}

We can now rule out affine monodromy on every non-hyperbolic Riemann
surface.

\begin{theorem}[Nonexistence of affine monodromy]
\label{thm:no-affine}
Let $\Sigma$ be a Riemann surface which is either compact or noncompact
and parabolic, and let $\dd s^2$ be a singular hyperbolic metric on
$\Sigma$ with closed discrete singular set. Then the monodromy group of
$\dd s^2$ is not contained in any conjugate of the affine subgroup $L$.
\end{theorem}

\begin{proof}
Suppose, to the contrary, that the monodromy group is contained in a
conjugate of $L$. After postcomposing the developing map with a suitable
element of $\PSL(2,\R)$, we may assume that the image of the chosen
monodromy representation is contained in $L$. Let $\alpha$ be the global
one-form constructed in Section~\ref{sec:differential-identity}.

Suppose first that $\Sigma$ is compact. Since the constant function
$\phi\equiv1$ belongs to $C_c^\infty(\Sigma)$,
Proposition~\ref{prop:energy-inequality} gives
$$
0
\leq
\int_{\Sigma^*}|\alpha|_g^2\,\dd A_g
\leq
4\int_{\Sigma}|\dd1|_g^2\,\dd A_g
=
0.
$$
Thus
$\int_{\Sigma^*}|\alpha|_g^2\,\dd A_g=0.$
Since $|\alpha|_g^2$ is continuous and nonnegative on $\Sigma^*$, it
follows that
$\alpha\equiv0$
on $\Sigma^*$.

Now suppose that $\Sigma$ is noncompact and parabolic. Let
$K_j$ and $\phi_j$ be given by
Lemma~\ref{lem:parabolic-cutoff}. Applying
Proposition~\ref{prop:energy-inequality} to $\phi_j$, we obtain
$$
\int_{\Sigma^*}
\phi_j^2|\alpha|_g^2\,\dd A_g
\leq
4\int_{\Sigma}
|\dd\phi_j|_g^2\,\dd A_g
<
\frac{4}{j}.
$$

Fix $x\in\Sigma^*$. Choose a relatively compact coordinate
neighborhood $V$ of $x$ such that
$\overline V\Subset\Sigma^*.$
For all sufficiently large $j$, one has
$\overline V\subset K_j,$
and hence $\phi_j\equiv1$ on $\overline V$. Therefore
$$
0
\leq
\int_V|\alpha|_g^2\,\dd A_g
\leq
\int_{\Sigma^*}
\phi_j^2|\alpha|_g^2\,\dd A_g
<
\frac{4}{j}.
$$
Letting $j\to\infty$ gives
$\int_V|\alpha|_g^2\,\dd A_g=0.$
Since $|\alpha|_g^2$ is continuous and nonnegative on $V$, it follows
that $\alpha=0$ on $V$. As $x\in\Sigma^*$ was arbitrary,
$\alpha\equiv0$
on $\Sigma^*$.

Thus in either case, $\alpha\equiv0.$
Pulling back to the universal cover yields
$\dd\log\widetilde v=\pi^*\alpha=0.$
Since the universal cover is connected,
$\log\widetilde v$, and hence $\widetilde v$, is constant. Writing
$\widetilde f=\widetilde u+{\rm i}\widetilde v,$
the Cauchy--Riemann equations imply that $\widetilde u$ is locally
constant and hence constant. Therefore $\widetilde f$ is constant,
contradicting the local univalence of the developing map.

This contradiction proves the theorem.
\end{proof}

Combining Theorem~\ref{thm:no-affine} with Proposition~\ref{prop:subgroup-reduction} and \cite[Theorem~1.6]{FengShiSongXu2023}, we conclude that the monodromy group of every singular hyperbolic metric on a non-hyperbolic Riemann surface is Zariski dense in $\PSL(2,\R)$. This proves Theorem~\ref{thm:main-intro}, and hence Conjecture~1.7 in \cite{FengShiSongXu2023}.

\section*{Acknowledgements}
The authors acknowledge ChatGPT’s assistance in the initial exploration of the proof of Theorem 1.2. Inspired by interactions with Y.F. in May 2026, they rigorously developed and verified the argument, assuming full responsibility for the paper.

Y.S. is supported in part by the National Natural Science Foundation of China (Grant No. 11931009) and the Innovation Program for Quantum Science and Technology (Grant No. 2021ZD0302902). 
J.S. is partially supported by the National Natural Science Foundation of China (Grant Nos. 12001399, 11831013 and 12171352) and the International Postdoctoral Exchange Fellowship Program by the Office of China Postdoctoral Council (No. PC2021053). 
B.X. is supported in part by the Project of Stable Support for Youth Team in Basic Research Field, CAS (Grant No. YSBR-001) and NSFC (Grant Nos. 12271495, 11971450 and 12071449).

\end{document}